\documentclass[11pt,a4paper,reqno]{amsart}
\usepackage[hidelinks]{hyperref} 
\pdfoutput=1
\usepackage{amssymb,interval,eulervm,palatino,fullpage,enumitem,amsaddr}
\usepackage[USenglish]{babel}
\usepackage[justification=centering,skip=5pt]{caption} 
\usepackage{tikz,tikz-cd}
\usetikzlibrary{arrows,arrows.meta,positioning}
\setenumerate{label=(\roman*),itemsep=2pt,topsep=2pt}

\allowdisplaybreaks[4]
\newcommand{\doi}[1]{\url{https://doi.org/#1}}
\newcommand{\isbn}[1]{\url{https://isbnsearch.org/isbn/#1}}
\newcommand{\arxiv}[1]{\href{https://arxiv.org/abs/#1}{preprint arXiv:#1}}
\newcommand{\web}[1]{\url{#1}}

\renewcommand{\emptyset}{\varnothing}
\numberwithin{equation}{section}

\newtheorem{thm}{Theorem}[section]
\newtheorem{lemma}[thm]{Lemma}
\newtheorem{prop}[thm]{Proposition}
\newtheorem{rem}[thm]{Remark}
\newtheorem{cor}[thm]{Corollary}
\newtheorem{df}[thm]{Definition}
\newtheorem{ex}[thm]{Example}

\newcommand{\A}{\mathcal{A}}

\newcommand{\Z}{\mathbb{Z}}
\newcommand{\N}{\mathbb{N}}

\newcommand{\mv}[1]{\smash[b]{\underline{#1}}}
\newcommand{\id}{\mathrm{id}}

\makeatletter
\def\@tocline#1#2#3#4#5#6#7{\relax
  \ifnum #1>\c@tocdepth 
  \else
    \par \addpenalty\@secpenalty\addvspace{#2}%
    \begingroup \hyphenpenalty\@M
    \@ifempty{#4}{%
      \@tempdima\csname r@tocindent\number#1\endcsname\relax
    }{%
      \@tempdima#4\relax
    }%
    \parindent\z@ \leftskip#3\relax \advance\leftskip\@tempdima\relax
    \rightskip\@pnumwidth plus4em \parfillskip-\@pnumwidth
    #5\leavevmode\hskip-\@tempdima
      \ifcase #1
       \or\or \hskip 1em \or \hskip 2em \else \hskip 3em \fi%
      #6 \hskip 0.5em \nobreak\relax
    \dotfill\hbox to\@pnumwidth{\@tocpagenum{#7}}\par
    \nobreak
    \endgroup
  \fi}
\makeatother

\begin{document}
\leftmargini=2em

\title{\vspace*{-1cm}The graph groupoid of an odd-dimensional \\[5pt] Vaksman--Soibelman quantum sphere}

\author[F.~D'Andrea]{\vspace*{-5mm}Francesco D'Andrea}

\address{Dipartimento di Matematica e Applicazioni ``R.~Caccioppoli'' \\ Universit\`a di Napoli Federico II
and INFN Sezione di Napoli \\
Complesso MSA, Via Cintia 26, 80126 Napoli, Italy}



\begin{abstract}
Quantum spheres are among the most studied examples of compact quantum spaces.
In particular, the $(2n+1)$-dimensional quantum sphere of Vaksman and Soibelman is 
described by a \mbox{C*-algebra} that is the Cuntz--Krieger algebra of a directed graph, as proved by Hong and Szyma{\'n}ski in 2002. A few years earlier, in 1997, Sheu proved that the same \mbox{C*-algebra} is isomorphic to the groupoid \mbox{C*-algebra} of an \'etale groupoid. Here we show that the path groupoid of the directed graph of Hong and Szyma{\'n}ski is isomorphic to the groupoid discovered by Sheu.
\end{abstract}

\maketitle

\vspace*{-2mm}

\begin{center}
\begin{minipage}{0.8\textwidth}
\parskip=0pt\small\tableofcontents
\end{minipage}
\end{center}

\medskip\smallskip

\section{Introduction}
Odd-dimensional quantum spheres were introduced in 1990 by Vaksman and Soibelman as quantum homogeneous spaces of Woronowicz's quantum unitary groups \cite{VS91}.
Since then, they have been a main example for testing ideas about noncommutative spaces, one of the reasons being their relation to quantum complex projective spaces (see e.g.~\cite{DL12} and references therein).
About half a decade after their discovery, Sheu proved that the \mbox{C*-algebra} of a quantum sphere is isomorphic to a groupoid \mbox{C*-algebra} \cite[Theorem~3.1]{She97}, and another half a decade later, Hong and Szyma{\'n}ski proved that it is a graph \mbox{C*-algebras} \cite{HS02} (see also \cite{Dan23} for an overview). The graph \mbox{C*-algebra} realization of quantum spheres is used, for example, to prove a pushout structure of quantum complex projective spaces \cite[Prop.~4.1]{ADHT22} that can be interpreted as a generalized CW-structure in the sense of \cite{DHMSZ20}.
The groupoid \mbox{C*-algebra} realization also has many applications: for example, it is used to compute the monoid of finitely generated projective modules over the \mbox{C*-algebra} of the quantum sphere and to prove that the cancellation property does not hold \cite[Cor.~1 and 2]{She19}.

The aim of this paper is to show that the path groupoid of the directed graph of Hong and 
Szyma{\'n}ski is isomorphic to the groupoid discovered by Sheu.
We emphasize that this result cannot simply be inferred from the fact that the two groupoids have isomorphic \mbox{C*-algebras} (non-isomorphic groupoids can have isomorphic groupoid \mbox{C*-algebras}: for example, the group \mbox{C*-algebra} of a finite abelian group depends only on its cardinality \cite{Dad71}). The existence of an isomorphism between two topological groupoids (second countable, locally compact, Hausdorff, étale, essentially principal) is equivalent to the existence of a \emph{diagonal-preserving} isomorphism between the associated reduced groupoid \mbox{C*-algebras} \cite[Thm.~3.3]{CRSM21}. 
In the purely algebraic setting, an analogous result holds for Steinberg algebras (see \cite[Thm.~5.6]{Ste19} for the precise statement). One of the interesting aspects of such isomorphisms is that they establish a connection between the purely algebraic and the operator algebraic settings.

\medskip

In this paper, we denote by $\N$ the set of natural numbers (including $0$) and by $\overline{\N}=\N\cup\{+\infty\}$ its one-point compactification. We denote by $S^1$ the set of unit complex numbers.
We adopt the convention that an empty sum is $0$.

\section{The graph C*-algebra of a quantum sphere}\label{sec:4}

Odd-dimensional quantum spheres were originally defined in \cite{VS91} as quantum homogeneous spaces of quantum unitary groups, and described by a set of ``non-commuting coordinates'' generating a *-algebra that we denote by $\A(S^{2n+1}_q)$. The odd integer $2n+1$ is the classical dimension and $q\in\ointerval{0}{1}$ is the deformation parameter. While the algebra $\A(S^{2n+1}_q)$ does depend on $q$, different values of $q$ corresponding to non-isomorphic algebras \cite[Thm.~1]{Dan24b}, the C*-enveloping algebra $C(S^{2n+1}_q)$ turns out to be independent of the deformation parameter \cite[Cor.~3.2]{She97} and isomorphic to a graph \mbox{C*-algebra} \cite{HS02}.

Let us recall the graph \mbox{C*-algebra} picture.
Our main references are \cite{BPRS,R05}. We adopt the conventions of \cite{BPRS}, i.e.~the roles of source and range maps are exchanged with respect to~\cite{R05}.

Let $E=(E^0,E^1,s,t)$ be a directed graph, where $E^0$ is the set of vertices, $E^1$ is the set of edges, \mbox{$s:E^1\to E^0$} is the source map,
and \mbox{$t:E^1\to E^0$} is the target (or range) map. 
A \emph{sink} is a vertex $v$ with no outgoing edges, that is $s^{-1}(v)=\emptyset$.
The graph is called \emph{finite} if both the sets $E^0$ and $E^1$ are finite. In this paper, we are only interested in finite graphs without sinks.

The \emph{graph \mbox{C*-algebra}} $C^*(E)$ of a finite graph $E$ without sinks is the universal \mbox{C*-algebra} generated by mutually 
orthogonal projections $\big\{P_v:v\in E^0\big\}$ and partial isometries $\big\{S_e:e\in E^1\big\}$ satisfying the \emph{Cuntz--Krieger relations}:
\begin{alignat}{2}
S_e^*S_e &=P_{t(e)} \qquad && \text{for all }e\in E^1\text{, and}\tag{\text{CK1'}} \label{eq:CK1p} \\
\sum_{e\in E^1:\,s(e)=v}\!\! S_eS_e^* &=P_v && \text{for all }v\in E^0 . \tag{CK2'} \label{eq:CK2p}
\end{alignat}
The existence of this universal object is a non-trivial theorem (see \cite{R05,BPRS}, also for the precise meaning of ``universal'' in this context).

Hong and Szyma{\'n}ski proved in their seminal paper that, for every $n\in\N$, $C(S^{2n+1}_q)$ is isomorphic to the \mbox{C*-algebra} of the graph with vertex set $\{1,\ldots,n+1\}$, and with one arrow from the vertex $i$ to the vertex $j$ for every $i\leq j$ \cite[Thm.~4.4]{HS02}; they denote this graph by $L_{2n+1}$. The same \mbox{C*-algebra} can be obtained from a different graph, denoted $\widetilde{L}_{2n+1}$, with the same vertex set $\{1,\ldots,n+1\}$, one loop $\ell_i$ at each vertex and one arrow $r_i$ from $i$ to $i+1$ for all $1\leq i\leq n$ (see \cite[Prop.~A.1]{HS02}). This is the graph that we will use, while the graph $L_{2n+1}$ will play no role in this paper. A picture of the two graphs is in Figure~\ref{fig:sphereA} and~\ref{fig:sphereB}.
The explicit isomorphism between $C^*(L_{2n+1})$ and
$C^*(\widetilde{L}_{2n+1})$ is given by a polynomial formula in the Cuntz-Krieger generators, see e.g.~\cite[Thm.~3.11]{Dan23}.

\begin{figure}[t]
\begin{tikzpicture}[>=stealth,node distance=2cm,
main node/.style={circle,inner sep=2pt},
freccia/.style={->,shorten >=2pt,shorten <=2pt},
ciclo/.style={out=130, in=50, loop, distance=2cm}]

\clip (-0.6,-2.7) rectangle (10.6,1.4);

\node[main node] (1) {};
\node (2) [main node,right of=1] {};
\node (3) [main node,right of=2] {};
\node (4) [main node,right of=3] {};
\node (5) [right of=4] {};
\node (6) [main node,right of=5] {};

\filldraw (1) circle (0.06) node[below left] {$1$};
\filldraw (2) circle (0.06);
\filldraw (3) circle (0.06);
\filldraw (4) circle (0.06);
\filldraw (6) circle (0.06) node[below right] {$n+1$};

\path[freccia] (1) edge[ciclo] (1);
\path[freccia] (2) edge[ciclo] (2);
\path[freccia] (3) edge[ciclo] (3);
\path[freccia] (4) edge[ciclo] (4);
\path[freccia] (6) edge[ciclo] (6);

\path[freccia]
	(1) edge (2)
	(2) edge (3)
	(3) edge (4)
	(4) edge[dashed]  (6)
	(1) edge[bend right] (3)
	(1) edge[bend right=40] (4)
	(2) edge[bend right] (4);

\path[white]
	(1) edge[bend right=60] coordinate (7) (6)
	(2) edge[bend right=50] coordinate (8) (6)
	(3) edge[bend right=40] coordinate (9) (6);

\path
	(1) edge[out=-60,in=180] (7)
	(2) edge[out=-50,in=180] (8)
	(3) edge[out=-40,in=180] (9);

\path[->,dashed,shorten >=2pt]
	(7) edge[out=0,in=240] (6)
	(8) edge[out=0,in=230] (6)
	(9) edge[out=0,in=220] (6);

\path[freccia,dashed] (4) edge[bend right,dashed] (6);
\end{tikzpicture}
\caption{The graph $L_{2n+1}$.}
\label{fig:sphereA}

\bigskip

\begin{tikzpicture}[>=stealth,node distance=2cm,
main node/.style={circle,inner sep=2pt},
freccia/.style={->,shorten >=2pt, shorten <=2pt},
ciclo/.style={out=50, in=130, loop, distance=2cm},
font=\small]

\clip (-0.6,-0.5) rectangle (10.4,1.4);

      \node[main node] (1) {};
      \node (2) [main node,right of=1] {};
      \node (3) [main node,right=2.7cm of 2] {};
      \node (4) [main node,right of=3] {};
      \node (5) [main node,right=2.7cm of 4] {};

      \filldraw (1) circle (0.06) node[below=2pt] {$1$};
      \filldraw (2) circle (0.06) node[below=2pt] {$2$};
      \filldraw (3) circle (0.06) node[below=2pt] {$i$};
      \filldraw (4) circle (0.06) node[below=2pt] {$i+1$};
      \filldraw (5) circle (0.06) node[below=2pt] {$n+1$};

      \path[freccia] (1) edge[ciclo] (1);
      \path[freccia] (2) edge[ciclo] (2);
      \path[freccia] (3) edge[ciclo] node[below] {$\,\ell_i$} (3);
      \path[freccia] (4) edge[ciclo] (4);
      \path[freccia] (5) edge[ciclo] (5);

      \path[freccia] (1) edge (2) (3) edge node[above] {$r_i$} (4);
      \path[freccia,dashed] (2) edge (3) (4) edge (5);
\end{tikzpicture}
\caption{The graph of $\widetilde{L}_{2n+1}$.}
\label{fig:sphereB}
\end{figure}

\section{From directed graphs to \'etale groupoids}\label{sec:7}
We want now to derive the groupoid picture of a quantum sphere in \cite{She97} from the graph \mbox{C*-algebra} picture.
For the realization of a graph \mbox{C*-algebra} as a groupoid \mbox{C*-algebra} we follow \cite{KPRR97}, but under the simplifying assumption that we have a finite graph, rather than just row-finite. We stress that \cite{KPRR97} uses the same conventions as \cite{BPRS} for graph \mbox{C*-algebras}, which are the ones adopted in this paper.

Isomorphic {\'e}tale groupoids are associated to isomorphic (full or reduced) groupoid \mbox{C*-algebras}. In the next sections we will construct an explicit isomorphism between the groupoid of the graph of $S^{2n+1}_q$ and Sheu's groupoid.
It is worth mentioning that in \cite[\S 7.3]{BCL14} the same groupoid of Sheu is derived using the framework of geometric quantization.

Let us start by recalling some basic constructions of groupoids, mainly to fix the notations.

\subsection{Basics of groupoids}
We assume that the reader is familiar with {\'e}tale groupoids, see e.g.~\cite{SSW20}.
A few words about our conventions are in order.
As costumary, we will write \mbox{$\mathcal{G}:\mathcal{G}^1\;\mathop{\rightrightarrows}\;\mathcal{G}^0$} to denote a groupoid, 
where $\mathcal{G}^1$ is the set of \emph{morphisms} and $\mathcal{G}^0$ is the set of \emph{objects}.
Some authors, e.g.~\cite{SSW20}, use the embedding $\mathcal{G}^0\to\mathcal{G}^1$ sending an object to the corresponding identity morphism to identify $\mathcal{G}^0$ with a subspace of $\mathcal{G}^1$. Here we follow \cite{KPRR97}, which does not use this identification.
The product $\alpha\beta$ of two elements $\alpha,\beta\in\mathcal{G}^1$ is defined whenever the target of $\alpha$ is the source of $\beta$.
A homomorphism between two groupoids \mbox{$\mathcal{G}:\mathcal{G}^1\;\mathop{\rightrightarrows}\;\mathcal{G}^0$} and $\mathcal{F}:\mathcal{F}^1\;\mathop{\rightrightarrows}\;\mathcal{F}^0$
is a pair of maps $\phi^0:\mathcal{G}^0\to\mathcal{F}^0$ and $\phi^1:\mathcal{G}^1\to\mathcal{F}^1$ commuting with source and target maps and compatible with the multiplication in the obvious sense. Since source and target maps are surjective,
$\phi^0$ is uniquely determined by $\phi^1$, and the compatibility with composition implies that $\phi^1$ automatically maps identity morphisms to identity morphisms and inverses to inverses. A homomorphism of groupoids is an isomorphism if and only if $\phi^1$ is bijective. A homomorphism of topological groupoids is a homomorphism of groupoids such that $\phi^0$ and $\phi^1$ are continuous, and it is an isomorphism of topological groupoids if and only if $\phi^0$ and $\phi^1$ are also homeomorphisms.

\smallskip

Given a groupoid $\mathcal{G}:\mathcal{G}^1\;\mathop{\rightrightarrows}\;\mathcal{G}^0$ and a subset $\mathcal{S}^0$ of $\mathcal{G}^0$, we can form the subgroupoid \mbox{$\mathcal{S}:\mathcal{S}^1\;\mathop{\rightrightarrows}\;\mathcal{S}^0$} whose morphisms are all the morphisms in $\mathcal{G}$ with both source and target in $\mathcal{S}^0$. Notice that $\mathcal{S}$ is a full subcategory of $\mathcal{G}$, hence the name \emph{full subgroupoid}, or also \emph{reduction} of $\mathcal{G}$ to the set of objects $\mathcal{S}^0$. If $\mathcal{G}$ is a topological groupoid, we will consider $\mathcal{S}$ as a topological groupoid with subspace topology on both morphisms and objects.

\smallskip

If $\mathcal{G}:\mathcal{G}^1\;\mathop{\rightrightarrows}\;\mathcal{G}^0$ and $\mathcal{F}:\mathcal{F}^1\;\mathop{\rightrightarrows}\;\mathcal{F}^0$ are two topological groupoids, one can form the Cartesian product
$\mathcal{G}\times\mathcal{F}:\mathcal{G}^1\times\mathcal{F}^1\;\mathop{\rightrightarrows}\;\mathcal{G}^0\times\mathcal{F}^0$, with product topology, and all structure maps (source, target, unit, multiplication, inverse) given component-wise. If $\mathcal{G}$ is a group, the set of objects $\mathcal{G}^0\times\mathcal{F}^0$ of the product groupoid can (and will) be identified with $\mathcal{F}^0$.

\smallskip

An equivalence relation $\mathcal{R}\subseteq X\times X$ on a set $X$ can be viewed as a groupoid $\mathcal{R}\;\mathop{\rightrightarrows}\;X$. Source and target of $(x,y)\in\mathcal{R}$ are $x$ and $y$, respectively, and one defines composition and inverse by $(x,y)\cdot (y,z)=(x,z)$ and $(x,y)^{-1}=(y,x)$. If $X$ is a topological space, we will refer to the topology on $\mathcal{R}$ induced by the product topology on $X\times X$ as the \emph{standard topology}.
A \emph{principal groupoid} is one of the form $\mathcal{R}\;\mathop{\rightrightarrows}\;X$ for some equivalence relation $\mathcal{R}$ on $X$, and some topology on $\mathcal{R}$ which is not necessarily the standard one.

\smallskip

Finally, if $\Gamma$ is a group acting on a set $X$, the \emph{transformation groupoid}
\[
\Gamma\ltimes X\;\mathop{\rightrightarrows}\;X
\]
has source and target maps, composition and inverse given by
\[
s(g,x):=gx , \qquad
t(g,x):=x , \qquad
(g,hx)\cdot (h,x) :=(gh,x) , \qquad
(g,x)^{-1}=(g^{-1},gx) .
\]
If $X$ is a second-countable locally compact Hausdorff space and $\Gamma$ a discrete group acting on $X$ by homeomorphisms, then the transformation groupoid is \'etale in the product topology.

Denote by $X\times_{X/\Gamma}X$ the equivalence relation on $X$ consisting of pair of points in the same $\Gamma$-orbit.
The \emph{canonical map}
\[
\Gamma\times X\longrightarrow X\times_{X/\Gamma}X , \qquad
(g,x)\longmapsto (gx,x) ,
\]
is a morphism of topological groupoids with respect to the standard topology on $X\times_{X/\Gamma}X$.
When the transformation groupoid is {\'e}tale, the canonical map is a local homeomorphism, since its components are the source and target maps, which are local homeomorphisms. If, in addition, $\Gamma$ acts freely on $X$, then the canonical map is bijective, which means that it is a (global) homeomorphism and hence an isomorphism of topological groupoids.

The main example of transformation groupoid, that we will use later on, is the following (in this example, the group action is not free).

\begin{ex}\label{ex:Toeplitzgroupoid}
Let $\overline{\Z}=\Z\cup\{+\infty\}$ be the one-point compactification of $\Z$, and consider the action of
$\Z^n$ on $\overline{\Z}^n$ by translations, with $k+\infty :=+\infty$ for all $k\in\Z$.
The corresponding transformation groupoid $\Z^n\ltimes\overline{\Z}^n$ is \'etale.
Its reduction $(\Z^n\ltimes\overline{\Z}^n)|_{\overline{\N}^n}$ is called the \emph{Toeplitz groupoid} of order $n$.
\end{ex}

\subsection{Graph groupoids}
Let $X$ be a second-countable locally compact Hausdorff space and \mbox{$\sigma:X\to X$} a local homeomorphism.
We will call the pair $(X,\sigma)$ a \emph{dynamical system}, and think of $\sigma$ as generating a (discrete) time evolution, i.e.~an action of the monoid $\N$ on $X$.
A special case is when $X$ is the space of right-infinite paths in a graph.

Let $E=(E^0,E^1,s,t)$ be a finite graph with no sinks. Let
\[
E^\infty:=\left\{x=(x_1,x_2,x_3,\ldots) \in\prod\nolimits_{i=1}^\infty E^1:t(x_i)=s(x_{i+1})\;\forall\;i\geq 1\right\}
\]
be the set of right-infinite paths in $E$. If $\alpha=(\alpha_1,\ldots,\alpha_j)$ is a finite path, we define $s(\alpha):=s(\alpha_1)$ and $t(\alpha):=t(\alpha_j)$. If $x\in E^\infty$ has $s(x):=s(x_1)=t(\alpha)$, we denote by
\[
\alpha x=(\alpha_1,\ldots,\alpha_j,x_1,x_2,x_3,\ldots) \in E^\infty
\]
the concatenation of $\alpha$ and $x$. We will think of a vertex as a path of length $0$. If $\alpha$ is a vertex $v$, we define $s(\alpha)=t(\alpha)=v$ to uniform the notations, and in this case by $\alpha x$ we mean simply the path $x$, whenever $x$ has source $v$. We denote by $|\alpha|$ the length of a finite path $\alpha$.

We put on $E^1$ the discrete topology, on $\prod_{i=1}^\infty E^1$ the Tychonoff topology, and on $E^\infty$ the corresponding subspace topology.
With this topology, $E^\infty$ is a totally-disconnected compact Hausdorff space. For every finite path $\alpha$ in $E$, define a \emph{cylinder set} $U(\alpha)$ as follows:
\begin{equation}\label{eq:cylinderset}
U(\alpha):=\big\{\alpha x :x\in E^\infty, s(x)=t(\alpha) \big\} .
\end{equation}
The collection of all cylinder sets is a basis (of clopen sets) for the topology of $E^\infty$. For every $x\in E^\infty$, any infinite subfamily of $\big\{U(x_1,\ldots,x_j):j\in\Z_+\big\}$ is a neighborhood basis of $x$.

The \emph{shift}
\begin{equation}\label{eq:localhom}
\sigma:E^\infty\to E^\infty \;,\qquad\sigma(x_1,x_2,x_3,\ldots):=(x_2,x_3,\ldots) ,
\end{equation}
is a local homeomorphism (it restricts to a homeomorphism on $U(\alpha)$ for all $\alpha$).

Given $x,y\in E^\infty$ and $k\in\Z$, if there exists $j\in\N$ such that $j+k\geq 0$ and $\sigma^j(x)=\sigma^{j+k}(y)$, we will say that $x$ and $y$ are \emph{shift equivalent} with \emph{lag} $k$ and write
\[
x\sim_ky .
\]
Being shift equivalent (for some $k\in\Z$) is an equivalence relation. Being shift equivalent with lag $0$ is a sub-relation, that we call \emph{tail equivalence}.

The \emph{graph groupoid} of $E$ is defined as follows \cite{KPRR97}. As an algebraic groupoid,
\[
\mathcal{G}^1:=\big\{(x,k,y)\in E^\infty\times\Z\times E^\infty:x\sim_k y \big\}
\;\mathop{\rightrightarrows}\;\mathcal{G}^0:=E^\infty \;.
\]
Source and target of $(x,k,y)$ are $x$ and $y$, respectively. The composition law is
\[
(x,j,y)\cdot (y,k,z):=
(x,j+k,z) .
\]
The inverse of $(x,k,y)$ is $(y,-k,x)$, and the element $(x,0,x)$ is the unit at $x$.
If $\alpha$ and $\beta$ are finite paths in $E$ with the same target, we define
\begin{equation}\label{eq:basicopenZ}
Z(\alpha,\beta):=\big\{ (\alpha x,|\beta|-|\alpha|,\beta x): x\in E^\infty, s(x)=t(\alpha)=t(\beta) \big\} .
\end{equation}
On $\mathcal{G}^1$ we consider the topology with basic open sets \eqref{eq:basicopenZ}.

\begin{rem}\label{rem:33}
With the topology generated by the sets \ref{eq:basicopenZ}, the groupoid $\mathcal{G}$ is second-countable locally compact, Hausdorff and \'etale \cite[Prop.~2.6]{KPRR97}. It is also amenable \cite[Cor.~5.5]{KPRR97}, which means that the reduced and full groupoid \mbox{C*-algebras} are isomorphic, and it is essentially principal \cite[Prop.~6.3]{KPRR97}, which means that we are in the framework of \cite[Thm.~3.3]{CRSM21}.
\end{rem}

It was proved in \cite{KPRR97} (for a more general class of graphs than those considered here) that:

\begin{thm}[\protect{\cite{KPRR97}}]\label{thm:isoEG}
If $\mathcal{G}$ is the graph groupoid of $E$, then $C^*(\mathcal{G})\cong C^*(E)$.
\end{thm}

If $\{P_v,S_e:v\in E^0,e\in E^1\}$ denotes the Cuntz--Krieger $E$-family in $C^*(E)$, the isomorphism in Theorem \ref{thm:isoEG} is explicitly given by
\[
P_v\longmapsto 1_{Z(v,v)}
\qquad
S_e\longmapsto 1_{Z(e,t(e))}
\]
where $1_S$ denotes the characteristic function of a set $S$, cf.~\cite[Prop.~4.1]{KPRR97}.

\smallskip

The following lemma will be useful later on.

\begin{lemma}\label{lemma32}
Assume that for each $x\in E^\infty$ there exists $j\in\N$ such that $\sigma^j(x)$ is a fixed point of $\sigma$. Then, there exists an idempotent map $\sigma^\infty$ from $E^\infty$ to the set of fixed points such that
\begin{enumerate}
\item\label{lemma32:i} $\sigma^j(x)=\sigma^\infty(x)$ for $j\gg 0$.
\item\label{lemma32:ii} $x,y\in E^\infty$ are shift equivalent $\iff$ they are tail equivalent $\iff\sigma^\infty(x)=\sigma^\infty(y)$.
\item\label{lemma32:iii} If $\sigma^\infty(x)=\sigma^\infty(y)$, then $x\sim_ky$ for all $k\in\Z$.
\end{enumerate}
\end{lemma}

\begin{proof}
Let $x\in E^\infty$, $j,k\in\N$, and let $\sigma^j(x)=p$ and $\sigma^k(x)=q$ be two fixed points of $\sigma$. Assume that $j\leq k$. Then
\[
q=\sigma^k(x)=\sigma^{k-j}\sigma^j(x)=\sigma^{k-j}(p)=p .
\]
Call $\sigma^\infty(x)$ the unique fixed point associated to $x$. Clearly $\sigma^\infty$ is idempotent, since $x=\sigma^\infty(x)$ for every fixed point $x$, and \ref{lemma32:i} is satisfied.

Assume that $\sigma^\infty(x)=\sigma^\infty(y)=:p$. Then, there exists $m\in\N$ such that $\sigma^i(x)=\sigma^j(y)=p$ for all $i,j\geq m$. Let $k\in\Z$. If $k\geq 0$, then $\sigma^{m}(x)=\sigma^{m+k}(y)=p$.
If $k<0$, then $\sigma^{m-k}(x)=\sigma^{m}(y)=p$. In both cases, $x\sim_ky$. This proves \ref{lemma32:iii}.

In particular, $\sigma^\infty(x)=\sigma^\infty(y)$ implies tail equivalence, and tail equivalence clearly implies shift equivalence.
It remains to prove that shift equivalence implies $\sigma^\infty(x)=\sigma^\infty(y)$.

Let $x,y\in E^\infty$, $k\in\Z$, and assume that $x\sim_ky$. Thus, there is $m_0\in\N$ such that $\sigma^{m}(x)=\sigma^{m+k}(y)$ for all $m\geq m_0$. For $m$ big enough, the left hand side of the latter equality is $\sigma^\infty(x)$ and the right hand side is $\sigma^\infty(y)$, hence the thesis.
\end{proof}

Under the assumption of previous lemma, there is an isomorphism of groupoids (but not necessarily of topological groupoids):
\begin{equation}\label{eq:wehaveabj}
\mathcal{G}^1\longrightarrow \Z\times\mathcal{R}_{\mathrm{tail}} , \qquad
(x,k,y)\longmapsto \big(k,(x,y)\big) ,
\end{equation}
where $\mathcal{G}$ is the graph groupoid and $\mathcal{R}_{\mathrm{tail}}:=\{(x,y):x\sim_0y\}$ is the relation of tail equivalence.

\begin{lemma}
Under the assumption of Lemma \ref{lemma32}, the map \eqref{eq:wehaveabj} is continuous.
\end{lemma}

\begin{proof}
Recall the notations \eqref{eq:cylinderset} and \eqref{eq:basicopenZ}.
A basis for the topology of $\Z\times\mathcal{R}_{\mathrm{tail}}$ is given by sets of the form
\[
\{k\}\times \Big(U(\alpha)\times U(\beta)\cap\mathcal{R}_{\mathrm{tail}}\Big)
\]
with $k\in\Z$ and $\alpha,\beta$ finite paths. Let $V$ be the preimage of this set under the map \eqref{eq:wehaveabj}.

Let $(x,k,y)\in V$.
By Lemma \ref{lemma32}, $(x,y)\in\mathcal{R}_{\mathrm{tail}}$ implies $x\sim_ky$, that means
$x=\alpha'z$ and $y=\beta'z$ for some finite paths $\alpha'$ and $\beta'$ with $|\beta'|-|\alpha'|=k$, and some $z\in E^\infty$.
We can choose $\alpha'$ and $\beta'$ such that $|\alpha'|\geq |\alpha|$ and $|\beta'|\geq |\beta|$ (by adding to both of them arrows from $z$ we do not change the difference $|\beta'|-|\alpha'|$).
Clearly
\[
(x,k,y)\in Z(\alpha',\beta') .
\]
But $x\in U(\alpha)$ and $y\in U(\beta)$ which means that $\alpha'_i=\alpha_i$ for all $i\leq|\alpha|$ and $\beta'_i=\beta_i$ for all $i\leq |\beta|$.

Take any point $(x',k,y')\in Z(\alpha',\beta')$. By construction , $x'=\alpha'z'$ and $y'=\beta'z'$ for some $z'\in E^\infty$, which means that $x'\in U(\alpha')\subseteq U(\alpha)$ and $y'\in  U(\beta')\subseteq U(\beta)$. From Lemma \ref{lemma32} we also have $(x',y')\in\mathcal{R}_{\mathrm{tail}}$. Thus, $(x',k,y')\in V$. This proves that
\[
(x,k,y)\in Z(\alpha',\beta')\subseteq V .
\]
Any point $(x,k,y)$ in $V$ is internal, so $V$ is open. Since $V$ is the preimage of an arbitrary basic open set, this proves that the map \eqref{eq:wehaveabj} is continuous.
\end{proof}

In general, the map \eqref{eq:wehaveabj} is not a homeomorphism: if we identify the underlying sets $\mathcal{G}^1$ and $\Z\times\mathcal{R}_{\mathrm{tail}}$, the former has a finer topology than the latter.
This happens for example with quantum spheres (see e.g.~Remark \ref{rem:finer} for $S^3_q$).

\smallskip

In the context of Lemma \ref{lemma32}, let $x\in E^\infty$ be a fixed point of $\sigma$. If $\sigma^\infty(y)=x$ for all $y$ in a neighborhood of $x$, we say that
$x$ is a \emph{stable equilibrium point} (and the maximal neighborhood of $x$ with this property is called its \emph{basin of attraction}). If $\sigma^\infty(y)\neq x$ for all $y$ in a neighborhood of $x$ with $y\neq x$,
we will say that $x$ is an \emph{unstable equilibrium point}.

\smallskip

We now specialize the discussion to the graph $\widetilde{L}_{2n+1}$ in Figure \ref{fig:sphereB} and its graph groupoid.
It follows from Theorem \ref{thm:isoEG} and the appendix of \cite{HS02} that the corresponding groupoid \mbox{C*-algebra} is isomorphic to $C(S^{2n+1}_q)$.
If $n=0$, the graph consists of a single loop
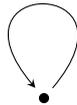
\begin{center}
\begin{tikzpicture}[>=stealth,node distance=2cm,
main node/.style={circle,inner sep=2pt},
freccia/.style={->,shorten >=2pt, shorten <=2pt},
ciclo/.style={out=50, in=130, loop, distance=2cm},
font=\small]

      \node[main node] (1) {};

      \filldraw (1) circle (0.06);

      \path[freccia] (1) edge[ciclo] (1);
\end{tikzpicture}
\end{center}
There is a unique right-infinite path in this graph, consisting of infinitely many loops around the only vertex of the graph. The graph groupoid is the group $\Z$, and its groupoid \mbox{C*-algebra} is $C^*(\Z)\cong C(S^1)$ as expected. Throughout the rest of this work, we study the case $n\geq 1$, starting with a detailed analysis of the path space.

\section{The path space of a quantum sphere}\label{sec:viathemap}
In this section $n\geq 1$, $E:=\widetilde{L}_{2n+1}$ is the graph in Figure~\ref{fig:sphereB}, and $\mathcal{G}$ is the corresponding graph groupoid. Recall that the vertex set is $E^0=\{1,\ldots,n+1\}$, and $E^1=\{\ell_1,\ldots,\ell_{n+1},r_1,\ldots,r_n\}$ consists of a loop $\ell_i$ at each vertex $i$,
and an edge $r_i$ from the vertex $i$ to the vertex $i+1$ for each $1\leq i\leq n$. We shall write a path $(x_1,x_2,x_3,\ldots)$ as a concatenation $x_1x_2x_3\cdots$ to save space.

Let $E^\infty_\ell\subset E^\infty$ be the subset of right-infinite paths whose first edge is a loop.
The map
\begin{equation}\label{eq:oneloop}
E^\infty\longrightarrow E^\infty_\ell , \qquad
x_1x_2x_3\cdots \longmapsto \ell_{s(x_1)}x_1x_2x_3\cdots ,
\end{equation}
which adds a loop at the beginning of the sequence, is a bijection with inverse
\[
\sigma:E^\infty_\ell \longrightarrow E^\infty
\]
given by the restriction to $E^\infty_\ell$ of the shift \eqref{eq:localhom}. In fact, \eqref{eq:oneloop} is a homeomorphism, since it transforms basic open sets \eqref{eq:cylinderset} into basic open sets.

Let $\overline{\N}^n_{\nearrow}$ denote the topological subspace of $\overline{\N}^n$ consisting of increasing tuples:
\[
\overline{\N}^n_{\nearrow}:=\big\{\mv{k}:=(k_1,\ldots,k_n)\in\overline{\N}^n:k_1\leq k_2\leq\ldots\leq k_n\big\} .
\]
We define a map
\begin{equation}\label{eq:mapF}
\Phi:E^\infty_\ell\longrightarrow\overline{\N}^n_{\nearrow}
\end{equation}
as follows. For any $x\in E_\ell^\infty$ and $1\leq i\leq n$, the integer $\Phi(x)_i$ counts the number of loops in the path $x$ before reaching the vertex $i+1$.
Let us give an explicit formula for $\Phi$.

Any $x\in E^\infty_\ell$  starts with a loop at some vertex $i_0$, does a number of loops around each vertex and then possibly moves forward, and ends with a tail consisting of infinitely many loops around some other vertex $i_1$.
Thus,
\begin{equation}\label{eq:foftheform}
x=\ell_{i_0}^{m_{i_0}}r_{i_0}\ell_{i_0+1}^{m_{i_0+1}}r_{i_0+1}\ell_{i_0+2}^{m_{i_0+2}}\ldots \ell_{i_1-1}^{m_{i_1-1}}r_{i_1-1}\ell_{i_1}^\infty
\end{equation}
for some $1\leq i_0\leq i_1\leq n+1$, $m_{i_0},\ldots,m_{i_1-1}\in\N$, $m_{i_0}\neq 0$. Here, $\ell_{i_1}^\infty$ means the loop $\ell_{i_1}$ repeated infinitely many times.
For $x$ as in \eqref{eq:foftheform},
\begin{equation}\label{eq:Phiofx}
\Phi(x)_i=\begin{cases}
0 & \text{if }i<i_0 , \\
\sum_{j=i_0}^im_j , & \text{if }i_0\leq i<i_1 \\
+\infty & \text{if }i\geq i_1 .
\end{cases}
\end{equation}
To simplify the discussion, we now define two maps
\[
\iota,\varepsilon:\overline{\N}^n\to\{1,\ldots,n+1\}
\]
as follows. Given any $\mv{k}=(k_1,\ldots,k_n)\in\overline{\N}^n$ set $k_{n+1}:=+\infty$ and define
\begin{equation}\label{eq:iotae}
\iota(\mv{k}):=\min\big\{i:k_i\neq 0\big\} , \qquad\quad
\varepsilon(\mv{k}):=\min\big\{i:k_i=+\infty\big\} .
\end{equation}
Observe that, thanks to $k_{n+1}=+\infty$, both maps are well-defined. We will call $\iota(\mv{k})$ the \emph{index} and
$\varepsilon(\mv{k})$ the  \emph{anchor} of $\mv{k}$, respectively.
If $x\in E^\infty_\ell$ and $\mv{m}=\Phi(x)$, then $i_0=\iota(\mv{m})$ coincides with the source of $x$, and $i_1=\varepsilon(\mv{m})$ is the vertex around which $x$ does infinitely many loops.

Clearly the map \eqref{eq:Phiofx} is bijective. The path $x$ can be uniquely reconstructed from the knowledge of the index $i_0$ and anchor $i_1$ of $\Phi(x)$ (i.e.\ the vertices where $x$ starts and ends with infinitely many loops, respectively), and from the number of loops $\Phi(x)_i-\Phi(x)_{i+1}$ around each vertex $i$, for $i_0\leq i<i_1$. Since we can loop around the first $n$ vertices as many times as we want, every tuple in $\overline{\N}^n_{\nearrow}$ is in the image of $\Phi$.

\begin{ex}\label{ex:SUq2}
Let us start with the special case $n=1$. In this case, the graph is:
\begin{center}
\begin{tikzpicture}[>=stealth,
every picture/.style=semithick,
node distance=2cm,
main node/.style={circle,inner sep=1pt,fill=gray!20},
freccia/.style={->,shorten >=2pt,shorten <=2pt},
ciclo/.style={out=50, in=130, loop, distance=2cm, ->},
inner sep=1pt,font=\small]

\node[main node] (1) {$1$};
\node (2) [main node,right of=1] {$2$};

\foreach \k in {1,2} \path[freccia] (\k) edge[ciclo] node[below=2pt] {$\ell_{\k}$} (\k);

\path[freccia]
	(1) edge node[above=2pt] {$r_1$} (2);

\end{tikzpicture}
\end{center}
Denote by
\[
\Psi:E^\infty\to\overline{\N}
\]
the composition of \eqref{eq:oneloop} and \eqref{eq:mapF}. Explicitly, for every $m\in\N$,
\[
\Psi(\ell_2^\infty):=0 , \qquad
\Psi(\ell_1^mr_1\ell_2^\infty):=m+1 , \qquad 
\Psi(\ell_1^\infty):=+\infty .
\]
Thus, $\Psi$ counts the number of edges in the path before reaching the vertex $2$, i.e.~$\Psi(x)$ is the length of the subpath obtained by removing from $x$ all the loops around the vertex $2$ (if any).
\end{ex}

\begin{ex}\label{ex:five}
Let $n=5$. In this case, the graph is:\vspace{-2pt}
\begin{center}
\begin{tikzpicture}[>=stealth,
every picture/.style=semithick,
node distance=2cm,
main node/.style={circle,inner sep=1pt,fill=gray!20},
freccia/.style={->,shorten >=2pt,shorten <=2pt},
ciclo/.style={out=50, in=130, loop, distance=2cm, ->},
inner sep=1pt,font=\small]

\node[main node] (1) {$1$};
\node (2) [main node,right of=1] {$2$};
\node (3) [main node,right of=2] {$3$};
\node (4) [main node,right of=3] {$4$};
\node (5) [main node,right of=4] {$5$};
\node (6) [main node,right of=5] {$6$};

\foreach \k in {1,...,6} \path[freccia] (\k) edge[ciclo] node[below=2pt] {$\ell_{\k}$} (\k);

\path[freccia]
	(1) edge node[above=2pt] {$r_1$} (2)
	(2) edge node[above=2pt] {$r_2$} (3)
	(3) edge node[above=2pt] {$r_3$} (4)
	(4) edge node[above=2pt] {$r_4$} (5)
	(5) edge node[above=2pt] {$r_5$} (6);

\end{tikzpicture}\vspace{-3pt}
\end{center}
One computes e.g.
\begin{gather*}
\Phi\big(\,
\ell_2^2r_2r_3\ell_4^{\infty}
\,\big)
=(0,2,2,\infty,\infty)
\\
\Phi\big(\,
\ell_6^{\infty}
\,\big)
=(0,0,0,0,0) 
\\
\Phi\big(\,
\ell_3r_3\ell_4r_4\ell_5^{\infty}
\,\big)
=(0,0,1,2,\infty)
 \\
\Phi\big(\,
\ell_1r_1r_2r_3\ell_4r_4r_5\ell_6^{\infty}
\,\big)
=(1,1,1,2,2)
\\
\Phi\big(\,
\ell_2^{\infty}
\,\big)
=(0,\infty,\infty,\infty,\infty)
\end{gather*}
\end{ex}

\begin{prop}
The map $\Phi$ in (\ref{eq:mapF}-\ref{eq:Phiofx}) is a homeomorphism.
\end{prop}

\begin{proof}
Since $\Phi$ is bijective, its domain is compact, and its codomain is Hausdorff, it is enough to show that $\Phi$ is continuous.
Recall that the topology of $E^\infty$ is generated by the sets \eqref{eq:cylinderset}.

Let $x=\alpha\ell_{n+1}^\infty$, where $\alpha$ is a finite path starting with a loop and with $t(\alpha)=n+1$.
Then $U(\alpha)$ is a singleton, since the only way to extend $\alpha$ to a right-infinite path is by adding infinitely many loops around the vertex $n+1$. In this case, $\{x\}$ is open, and $\Phi$ is continuous at $x$.

Next, let $x=\alpha r_{i_1-1}\ell_{i_1}^\infty$ with $i_1\leq n$ and $t(\alpha)=i_1-1$. Then, $\mv{k}:=\Phi(x)$ is of the form
\begin{equation}\label{eq:riciclaA}
\mv{k}=(k_1,\ldots,k_{i_1-1},+\infty,\ldots,+\infty)
\end{equation}
with $k_1,\ldots,k_{i_1-1}$ finite. A neighborhood basis of $\mv{k}$ in $\overline{\N}^n_{\nearrow}$ is given by the sets
\begin{equation}\label{eq:riciclaB}
V_p:=\Big(\{k_1\}\times\ldots\times\{k_{i_1-1}\}\times\interval{p}{+\infty}^{n-i_1+1} \Big)\cap\,\overline{\N}^n_{\nearrow}
\end{equation}
with $p\geq k_{i_1-1}$. A path in $\Phi^{-1}(V_p)$ loops around each vertex from $1$ to $i_1-1$ as many times as $\alpha$,
loops around the vertex $i_1$ at least $p-k_{i_1-1}$ times, and does an arbitrary number of loops around the other vertices. Thus,
\[
\Phi^{-1}(V_p)=\big\{ \alpha\ell_{i_1}^{p-k_{i_1-1}} y:y\in E^\infty \big\}=U\big(\alpha\ell_{i_1}^{p-k_{i_1-1}}\big) .
\]
This set is clearly open, hence $\Phi$ is continuous.
\end{proof}

We already saw three different realizations of the same topological space, as $E^\infty$, as $E^\infty_\ell$, and as
$\overline{\N}^n_{\nearrow}$. We now describe one more realization that will be useful later on to compare the path groupoid of $E$ with the groupoids in \cite{She97bis,She97}.

\smallskip

Define the following equivalence relation $\mathcal{R}_\infty$ on $\overline{\N}^n$:
\begin{equation}\label{eq:Rinfinity}
(\mv{m},\mv{m}')\in\mathcal{R}_\infty \iff \varepsilon(\mv{m})=\varepsilon(\mv{m}')=:i_1\text{ and }m_i=m'_i\;\forall\;i<i_1.
\end{equation}
Equip $\overline{\N}^n\!/\mathcal{R}_\infty$ with the quotient topology. Then,

\begin{prop}\label{prop:morup}
The map
\begin{equation}\label{eq:partialsums}
\overline{\N}^n\longrightarrow\overline{\N}^n_{\nearrow}, \qquad
\mv{m}\longmapsto (m_1,m_1+m_2,m_1+m_2+m_3,\ldots) ,
\end{equation}
induces a homeomorphism between $\overline{\N}^n\!/\mathcal{R}_\infty$ and $\overline{\N}^n_{\nearrow}$.
\end{prop}

\begin{proof}
Consider the following diagram
\[
\begin{tikzpicture}[>=To,xscale=3,yscale=2]

\node (a) at (0,1) {$\overline{\N}^n$};
\node (b) at (1,1) {$\overline{\N}^n_{\nearrow}$};
\node (c) at (0,0) {$\overline{\N}^n\!/\mathcal{R}_\infty$};

\draw[font=\footnotesize,->]
	(c) edge[dashed] node[below right,pos=0.45] {$\widetilde{F}$} (b)
	(a) edge node[above] {$F$} (b)
	(a) edge node[left,pos=0.45] {$\pi$} (c);

\end{tikzpicture}
\]
where $F$ is the map \eqref{eq:partialsums} and $\pi$ is the quotient map.
We first prove that $F$ and $\pi$ have the same fibers, so that by the universal property of quotient maps there exists a unique map $\widetilde{F}:\overline{\N}^n\!/\mathcal{R}_\infty\to\overline{\N}^n_{\nearrow}$ that makes the diagram commute, and such a map is bijective.

Let $\mv{m},\mv{m}'\in\overline{\N}^n$ have anchors $i$ and $i'$, respectively.

If $[\mv{m}']=[\mv{m}]$, then for all $1\leq j\leq n$ we have
\begin{equation}\label{eq:infinf} 
F(\mv{m})_j=\sum_{k=1}^jm_j=\sum_{k=1}^jm'_j=F(\mv{m}')_j .
\end{equation}
If $j<i=i'$ this follows from $m'_k=m_k$ for all $k\leq j<i$. If $j\geq i$ this follows from $m_i=m_i'=+\infty$, which makes both sums in \eqref{eq:infinf} equal to $+\infty$.

Conversely, suppose that $F(\mv{m})=F(\mv{m}')$ and assume, without loss of generality, that $i\geq i'$.
Then $m_1=F(\mv{m})_1=F(\mv{m}')_1=m'_1$. For each $2\leq j<i'$,
\[
m_j=F(\mv{m})_j-F(\mv{m})_{j-1}=F(\mv{m}')_j-F(\mv{m}')_{j-1}=m'_j .
\]
Finally, since $F(\mv{m})_{i'}=F(\mv{m}')_{i'}=m'_{i'}+F(\mv{m}')_{j-1}=+\infty$ and $F(\mv{m})_{i'-1}$ is finite, it must be
$m_{i'}=+\infty$. Thus, $i=i'$ and $[\mv{m}]=[\mv{m}']$.

\medskip

Next, we will show that $F$ is continuous, so that by the universal property of the quotient topology $\widetilde{F}$ is continuous as well.
Since the domain of $\widetilde{F}$ is compact and the codomain is Hausdorff, from the continuous bijection theorem it will follow that $\widetilde{F}$ is a homeomorphism.

The map $F$ is clearly continuous at each point in $\N^n$, since these points are open in $\overline{\N}^n$.

Now, assume that $\mv{m}\in\overline{\N}^n$ has anchor $i_1\leq n$, and call $\mv{k}:=F(\mv{m})$.
Thus, $\mv{k}$ is of the form \eqref{eq:riciclaA} and has neighborhood basis $\big\{V_p:p\geq k_{i_1-1}\big\}$,
with $V_p$ given by \eqref{eq:riciclaB}.
Let
\begin{equation}\label{eq:opensat}
U_p:=\{m_1\}\times\ldots\times\{m_{i_1-1}\}\times\interval{p-k_{i_1-1}}{+\infty}\times\overline{\N}^{n-i_1} .
\end{equation}
Clearly $U_p$ is open in $\overline{\N}^n$ and $F(U_p)\subseteq V_p$. This proves that $F$ is continuous at $\mv{m}$.
\end{proof}

\section{The groupoid of the quantum $SU(2)$ group}\label{sec:suq2}
Now that we have a nice description of the path space, we pass to the graph groupoid of a quantum sphere $S^{2n+1}_q$. As a warm up, we start with the case $n=1$, i.e.~the quantum space underlying the quantum group $SU_q(2)$. We want to establish an isomorphism between the groupoid of the graph in Example \ref{ex:SUq2} and the groupoid in \cite{She97bis}.

We will use the map $\Psi$ in Example \ref{ex:SUq2} to identify the path space $E^\infty$ with $\overline{\N}$. Under this identification, the local homeomorphism in \eqref{eq:localhom} becomes the map $\sigma:\overline{\N}\to\overline{\N}$ given by
\[
\sigma(m)=\begin{cases}
0 & \text{if } m=0 ,\\
m-1 & \text{if } 1\leq m<+\infty ,\\
+\infty & \text{if } m=+\infty .
\end{cases}
\]
We see that the dynamical system $(\overline{\N},\sigma)$ has two fixed points, $0$ and $+\infty$. The first one is a stable equilibrium point: every $m\in\N$ ends at $0$ after a sufficiently long time. The second one is an unstable equilibrium point: every neighborhood of $+\infty$ contains a point in the basin of attraction of $0$. In particular, this dynamical system is not asymptotically stable: a small perturbation of $+\infty$ changes the long time limit of the dynamics.

We are in the situation illustrated in Lemma \ref{lemma32}. Two paths are shift equivalent if and only if they have the same tail.
Under the above identification, $m,m'\in\overline{\N}$ are tail equivalent if they are both finite, or they are both infinite.

Let $\mathcal{G}$ be the graph groupoid.
We can identify $\mathcal{G}^1$ with $\Z\times\mathcal{R}_{\mathrm{tail}}$ (as a set, not as a topological space) using \eqref{eq:wehaveabj}. Recall that $\mathcal{R}_{\mathrm{tail}}$ is 
the set of pairs of right infinite paths that are tail equivalent. Under the above identification,
\[
\mathcal{R}_{\mathrm{tail}}=\N^2\cup\{(+\infty,+\infty)\} .
\]
The graph groupoid is
\begin{equation}\label{eq:identification}
\mathcal{G}^1=\Z\times\mathcal{R}_{\mathrm{tail}}
\;\mathop{\rightrightarrows}\;
\mathcal{G}^0 = \overline{\N} .
\end{equation}
The topology of $\mathcal{G}^0$ is the obvious one, while the topology of $\mathcal{G}^1$ is described in the next lemma.

\begin{lemma}\label{Zjk}
A basis for the topology of $\mathcal{G}^1$ is given by the singletons in $\Z\times\N^2$, and by the sets
\[
Z_{k,m}:=\big\{ (-k,j+k,j): j\in\overline{\N}, j\geq m \big\}
\]
for all $k\in\Z$ and $m\in\N$ such that $m+k\geq 0$.
\end{lemma}

\begin{proof}
The topology of the graph groupoid is generated by the sets $Z(\alpha,\beta)$ in \eqref{eq:basicopenZ}.
 
Let $\alpha,\beta$ be finite paths that end at the vertex $2$. This means that $\alpha$ consists of $a\in\N$ edges before reaching the vertex $2$, and $a'\in\N$ loops around the vertex $2$ ($\alpha$ is a vertex if $a=a'=0$).
Similarly for $\beta$, with $b$ the number of edges before the vertex $2$ and $b'$ the number of loops around $2$. The only right-infinite path that we can attach to $\alpha$ or $\beta$ is the infinite cycle $\ell_2^\infty$.
The corresponding set $Z(\alpha,\beta)$ is a singleton, and under the identification \eqref{eq:identification}
becomes the point
\[
\{ (b+b'-a-a' ,a, b) \}
\]
In this way we can get any point in $\Z\times\N^2$, hence these points are all open in $\mathcal{G}^1$.

Next, suppose $\alpha,\beta$ are finite paths that end at the vertex $1$, which means $\alpha=\ell_1^a$ and $\beta=\ell_1^b$
for some $a,b\in\N$ (with $a=0$ resp.~$b=0$ corresponding to a path of length $0$, i.e.~the vertex $1$).
Since we can attach to them any path of the form $\ell_1^pr_1\ell_2^\infty$ or $\ell_1^\infty$, we get
\[
Z(\alpha,\beta)=\big\{ (\ell_1^{a+p}r_1\ell_2^\infty,b-a,\ell_1^{b+p}r_1\ell_2^\infty): p\in \N \big\} \cup 
\big\{ (\ell_1^\infty,b-a,\ell_1^\infty) \big\}.
\]
Under the above identification, this becomes the set
\[
\big\{ (b-a,a+p+1,b+p+1): p\in\overline{\N} \big\} 
\]
With the relabeling $k=a-b$, $j=b+p+1$ and $m=b+1$, we get the set $Z_{k,m}$
for every $k\in\Z$ and $m\in\Z_+$ such that $m+k\geq 1$.
On the other hand, for every $k\in\Z$ and $m\in\N$ such that $m+k\geq 0$, one has
\[
Z_{k,m}=Z_{k,m+1}\cup\big\{ (-k,m+k,m) \big\}
\]
and since the sets on the right hand side are both open, $Z_{k,m}$ is open as well.
\end{proof}

\begin{rem}\label{rem:finer}
The topology of $\mathcal{G}^1$ is strictly finer than the standard topology of $\Z\times\mathcal{R}_{\mathrm{tail}}$.
Indeed, the set $Z_{k,m}$ of Lemma \ref{Zjk} is not open in the standard topology of $\Z\times\mathcal{R}_{\mathrm{tail}}$.
By contradiction, if it were, the point $(-k,+\infty,+\infty)$ would be an interior point of $Z_{k,m}$. Hence, for $p$ big enough, we should have
\[
\{-k\}\times\big(\interval{p}{+\infty}^2\cap \mathcal{R}_{\mathrm{tail}}\big)
\subseteq Z_{k,m} ,
\]
which is clearly impossible.
\end{rem}

\begin{ex}\label{ex:groupF1}
Let $(\Z\ltimes\overline{\Z})|_{\overline{\N}}$ be the Toeplitz groupoid of order $1$ (cf.~Example \ref{ex:Toeplitzgroupoid}).
Let $\Z$ be the group of integers, regarded as a groupoid, and let 
$\Z\times (\Z\ltimes\overline{\Z})|_{\overline{\N}}$ denote the product groupoid.
Define $\mathcal{F}$ as the subgroupoid of $\Z\times (\Z\ltimes\overline{\Z})|_{\overline{\N}}$ consisting of morphisms
$(k_0,k_1,m_1)$ satisfying
\[
m_1=+\infty\;\Longrightarrow\;k_0+k_1=0 ,
\]
One can verify that $\mathcal{F}$ is indeed a subgroupoid, but not a full subgroupoid
since it has all the objects of $\Z\times (\Z\ltimes\overline{\Z})|_{\overline{\N}}$ but not all the morphisms.
\end{ex}

With the exception of slightly different notation ($k_0$ replaced by $-k_0$), the groupoid $\mathcal{F}$ defined in Example \ref{ex:groupF1} coincides with the one defined in the penultimate paragraph of \cite[Sect.~2]{She97bis}.

\begin{prop}\label{prop:comb}
The map
\begin{equation}\label{finalisoSUq2}
\Z\times (\Z\ltimes\overline{\Z})|_{\overline{\N}}\longrightarrow \mathcal{G}^1 , \qquad
(k_0,k_1,m_1)\longmapsto (k_0,k_1+m_1,m_1)
\end{equation}
induces an isomorphism of topological groupoids between the groupoid $\mathcal{F}$ in Example \ref{ex:groupF1}
and the graph groupoid $\mathcal{G}$ of the graph in Example \ref{ex:SUq2}.
\end{prop}

\begin{proof}
The groupoid $\mathcal{F}$ has morphisms
$(k_0,k_1,m_1)\in\Z^2\times\N$ with $m_1+k_1\geq 0$ and morphisms
$(-k_1,k_1,+\infty)$ with $k_1\in\Z$. The map \eqref{finalisoSUq2} gives a bijection between the first class of morphisms and the morphisms in $\Z\times\N^2\subseteq\mathcal{G}^1$, and these are open points (of the domain and codomain, respectively).
It also gives a bijection between the second class of morphisms and the morphisms in $\Z\times\{(+\infty,+\infty)\}\subseteq\mathcal{G}^1$. For the second class of morphisms, a neighborhood basis of $(-k_1,k_1,+\infty)$ is given by the family of sets
\[
\{-k_1\}\times\{k_1\}\times\interval{m_1}{+\infty}
\]
for all $m_1\geq\max\{-k_1,0\}$. The map \eqref{finalisoSUq2} transforms this set into the set
$Z_{m_1,k_1}$ of Lemma~\ref{Zjk}. Since it transforms a local basis into a local basis, it is a homeomorphism.
\end{proof}

Combining Theorem \ref{thm:isoEG} and Prop.~\ref{prop:comb} we 
get an independent proof of the isomorphism stated in the penultimate paragraph of \cite[Sect.~2]{She97bis} (this is also a special case of \cite[Thm.~3.1]{She97}).

\begin{cor}[\cite{She97bis}]
$C(S^3_q)\cong C^*(\mathcal{F})$, where $\mathcal{F}$ is the groupoid defined in Example \ref{ex:groupF1}.
\end{cor}

\section{The graph groupoid of a quantum sphere}

We now pass to the general case. As discussed at the end of Section \ref{sec:7}, we consider $n\geq 1$.

\begin{ex}\label{ex:reduction}
Let $(\Z^n\ltimes\overline{\Z}^n)|_{\overline{\N}^n}$ be the Toeplitz groupoid of order $n$ (cf.~Example \ref{ex:Toeplitzgroupoid}).
Explicitly, morphisms in this groupoid are pairs
\begin{equation}\label{eq:ZZnNn}
(\mv{k},\mv{m})\in\Z^n\times\overline{\N}^n
\end{equation}
such that $\mv{k}+\mv{m}\in\overline{\N}^n$. The morphism \eqref{eq:ZZnNn} has source $\mv{k}+\mv{m}$, target $\mv{m}$, and composition and inverse are given by
\[
(\mv{k}',\mv{k}+\mv{m})\cdot (\mv{k},\mv{m})=(\mv{k}+\mv{k}',\mv{m}) , \qquad
(\mv{k},\mv{m})^{-1}=(-\mv{k},\mv{k}+\mv{m}) .
\]
\end{ex}

We extend the equivalence relation $\mathcal{R}_\infty$ in \eqref{eq:Rinfinity} to the Cartesian product $\Z\times\Z^n\times\overline{\N}^n$ in the obvious way, and write
\begin{equation}\label{eq:relF}
(k_0,\mv{k},\mv{m})\sim (k_0',\mv{k}',\mv{m}') \iff k_0=k_0',\mv{k}=\mv{k}'\text{ and }
(\mv{m},\mv{m}')\in\mathcal{R}_\infty
\end{equation}
for all $(k_0,\mv{k},\mv{m})$ and $(k_0',\mv{k}',\mv{m}')$ in $\Z\times\Z^n\times\overline{\N}^n$.

\begin{ex}\label{ex:sheu}
Let $\widetilde{\mathfrak{F}}_n$ be the subgroupoid of $\Z\times (\Z^n\ltimes\overline{\Z}^n)|_{\overline{\N}^n}$ whose morphisms $(k_0,\mv{k},\mv{m})$ satisfy
\begin{equation}\label{eq:fixes}
\sum_{i=0}^{i_1}k_i=0\;\text{ and }\;k_j=0\;\forall\;i_1<j\leq n
\end{equation}
whenever $i_1:=\varepsilon(\mv{m})\leq n$.
We denote by $\mathfrak{F}_n$ the quotient groupoid by the relation \eqref{eq:relF}. This has
set of morphisms $\mathfrak{F}_n^1:=\widetilde{\mathfrak{F}}_n^1/{\sim}$ and set of objects $\mathfrak{F}_n^0:=\widetilde{\mathfrak{F}}_n^0/\mathcal{R}_\infty=\overline{\N}^n\!/\mathcal{R}_\infty$.
\end{ex}

It is not difficult to check that $\widetilde{\mathfrak{F}}_n$ in Example \ref{ex:sheu} is a subgroupoid, and that the composition of morphisms is compatible with the equivalence relation \eqref{eq:relF}, thus inducing a groupoid structure on the quotient. We leave this verification to the reader.

The groupoid $\mathfrak{F}_n$ was introduced by Sheu in \cite{She97}. Its space of units $\mathfrak{F}_n^0$ is homeomorphic to the path space of the graph $E:=\widetilde{L}_{2n+1}$ of $S^{2n+1}_q$ in Figure \ref{fig:sphereB}. The aim of this section is to show that $\mathfrak{F}_n$ is isomorphic to the graph groupoid of $E$, that in the following we denote by $\mathcal{G}$.

First, we need to introduce another groupoid. Similarly to Example \ref{ex:reduction} one can consider the reduction  $(\Z^n\ltimes\overline{\Z}^n)|_{\overline{\N}^n_{\nearrow}}$ of the transformation groupoid $\Z^n\ltimes\overline{\Z}^n$ to 
the set of objects $\overline{\N}^n_{\nearrow}$, the difference with Example \ref{ex:reduction} being that now $(\mv{k},\mv{m})\in\Z^n\times\overline{\N}^n_{\nearrow}$ must satisfy $\mv{k}+\mv{m}\in\overline{\N}^n_{\nearrow}$.

\begin{df}\label{ex:sheudf}
We denote by $\mathcal{F}$ the subgroupoid of $\Z\times (\Z^n\ltimes\overline{\Z}^n)|_{\overline{\N}_{\nearrow}^n}$ with the same set of objects, and with morphisms $(k_0,\mv{k},\mv{m})$ satisfying the additional condition
\begin{equation}\label{eq:fixesbis}
k_j=-k_0\;\forall\;i_1\leq j\leq n
\end{equation}
whenever $i_1:=\varepsilon(\mv{m})\leq n$.
\end{df}

\begin{prop}
For $\mv{k}$ in either $\Z^n$ or $\overline{\N}^n$, let
\begin{equation}\label{eq:mapf}
f(\mv{k}):=(k_1,k_1+k_2,k_1+k_2+k_3,\ldots,k_1+\ldots+k_n) .
\end{equation}
Then, the map
\[
\mathfrak{F}_n^1\longrightarrow\mathcal{F}^1, \qquad
(k_0,\mv{k},[\mv{m}])\longmapsto \big(k_0,f(\mv{k}),f(\mv{m})\big) ,
\]
is an isomorphism of topological groupoids.
\end{prop}

\begin{proof}
With an abuse of notations we denote by $f$ the map \eqref{eq:mapf} whatever is the domain considered.
As a map $\Z^n\to\Z^n$, $f$ is bijective (thus, a homeomorphism). As a map $\overline{\N}^n\to\overline{\N}^n_{\nearrow}$,
$f$ is the map \eqref{eq:partialsums} inducing a homeomorphism $\overline{\N}^n\!/\mathcal{R}_\infty\to\overline{\N}^n_{\nearrow}$. The map
\[
F:\Z\times\Z^n\times (\overline{\N}^n\!/\mathcal{R}_\infty)\longrightarrow
\Z\times\Z^n\times\overline{\N}^n_{\nearrow}\;
, \qquad
(k_0,\mv{k},[\mv{m}])\longmapsto \big(k_0,f(\mv{k}),f(\mv{m})\big) ,
\]
is then well defined and a homeomorphism.

Now, $\mathfrak{F}_n^1$ is a topological subspace of 
$\Z\times\Z^n\times (\overline{\N}^n\!/\mathcal{R}_\infty)$, and $\mathcal{F}^1$ is a topological subspace of $\Z\times\Z^n\times\overline{\N}^n_{\nearrow}$. We must show that $F$ maps the former subspace into the latter.

Clearly, $\mv{k}+\mv{m}\in\overline{\N}^n$ if and only if $f(\mv{k}+\mv{m})=f(\mv{k})+f(\mv{m})\in\overline{\N}^n_{\nearrow}$.

Let $(k_0,\mv{k},\mv{m})\in\Z\times (\Z^n\ltimes\overline{\Z}^n)|_{\overline{\N}^n}$ and call $\mv{k}'=f(\mv{k})$ and $\mv{m}'=f(\mv{m})$. Observe that $\mv{m}$ and $\mv{m}'$ have the same anchor, say $i_1$. Assume that $i_1\leq n$.
Since $\sum_{j=1}^{i_1}k_j=k_{i_1}'$, the first condition in \eqref{eq:fixes} is equivalent to $k_0+k_{i_1}'=0$.
Moreover, for all $i_1<j\leq n$,
\[
k_j=k'_j-k'_{j-1}=0 \iff k'_j=k'_{j-1} ,
\]
which together with $k_{i_1}'=-k_0$ proves that $F$ transforms indeed the constraint \eqref{eq:fixes} into \eqref{eq:fixesbis}.

Finally, we check that the restriction $F:\mathfrak{F}_n^1\to\mathcal{F}^1$ is a homomorphism of groupoids. Firstly,
it sends units $(0,0,[\mv{m}])$ into units $(0,0,f(\mv{m}))$. Moreover:
\begin{multline*}
F\big((k'_0,\mv{k}',[\mv{k}+\mv{m}])\cdot
(k_0,\mv{k},[\mv{m}])\big)=
F\big(k_0+k'_0,\mv{k}+\mv{k}',[\mv{m}]\big) \\
=\big(k_0+k'_0,f(\mv{k})+f(\mv{k}'),f(\mv{m})\big)
=\big(k'_0,f(\mv{k}'),f(\mv{k})+f(\mv{m})\big)\cdot
\big(k_0,f(\mv{k}),f(\mv{m})\big) \\
=F(k'_0,\mv{k}',[\mv{k}+\mv{m}])\cdot
F(k_0,\mv{k},[\mv{m}])
\end{multline*}
for any couple of morphisms $(k'_0,\mv{k}',[\mv{k}+\mv{m}])$ and $(k_0,\mv{k},[\mv{m}])$ in $\mathfrak{F}_n^1$.
\end{proof}

To finish, we now show that the topological groupoid $\mathcal{F}$ in Def.~\ref{ex:sheudf} is isomorphic to the graph groupoid $\mathcal{G}$.
For starters, we identify $\mathcal{G}^0=E^\infty$ with $\overline{\N}^n_{\nearrow}$ via the homeomorphism in Sect.~\ref{sec:viathemap}.
With this identification, the bijection \eqref{eq:wehaveabj} becomes
\[
\mathcal{G}^1 \longrightarrow \Z\times\mathcal{R}_{\mathrm{tail}}=\Big\{(k_0,\mv{m}',\mv{m})\in\Z\times (\overline{\N}^n_{\nearrow})^2: \varepsilon(\mv{m})=\varepsilon(\mv{m}') \Big\} .
\]
We will identify $\mathcal{G}^1$ with the set on the right hand side of this equality. Let us denote by
$\Z\times\mathcal{R}_{\mathrm{tail}}^{\mathrm{finite}}$ the set of morphisms $(k_0,\mv{m}',\mv{m})$ such that $\mv{m}'$ and $\mv{m}$ have finite components, that means $\varepsilon(\mv{m})=\varepsilon(\mv{m}')=n+1$.

\begin{rem}
Call $\Gamma:=\Z^n$, $X:=\overline{\Z}^n$ and consider the map
\begin{equation}\label{eq:Zcan}
\Z\times (\Gamma\times X)\xrightarrow{\quad\id_{\Z}\times\mathrm{can}\quad}
\Z\times (X\times_{X/\Gamma}X)
\end{equation}
where $\mathrm{can}:\Gamma\times X\to X\times_{X/\Gamma}X$ is the canonical map. If we consider the standard topology on $X\times_{X/\Gamma}X$, then \eqref{eq:Zcan} is a homomorphism of topological groupoids. One can check that this map, by restriction and corestriction, gives an isomorphism $F:\mathcal{F}^1\to\mathcal{G}^1$ of algebraic groupoids.
However, since our topology on $\mathcal{G}^1=\Z\times\mathcal{R}_{\mathrm{tail}}\subseteq \Z\times (X\times_{X/\Gamma}X)$ is finer than the standard topology, we cannot conclude that $F$ is continuous.
If the action of $\Gamma$ on $X$ were free, $F^{-1}$ would be a restriction of $\id_{\Z}\times\mathrm{can}^{-1}$, which being continuous for the standard topology on $\Z\times\mathcal{R}_{\mathrm{tail}}$, would be continuous for our finer topology as well. But the action is not free. We are in the worst-case scenario. In fact, we will see (cf.~Prop.~\ref{thm:final}) that one has to modify the map \eqref{eq:Zcan} a little bit to get the desired isomorphism of topological groupoids.
\end{rem}

Let $(k_0,\mv{k},\mv{m})\in\mathcal{F}^1$, assume that $i_1:=\varepsilon(\mv{m})\leq n$, call $\delta=\iota(\mv{m})-\iota(\mv{k}+\mv{m})$ the \emph{offset}, and let $p\geq m_{i_1-1}$ be a positive integer. Define a subset of $\mathcal{G}^1$ as follows:
\begin{equation}\label{eq:thesetZk0}
Z(k_0,\mv{k},\mv{m},p):=
\Big\{ (k_0+\delta,\mv{k}+\mv{m}',\mv{m}')\in\mathcal{G}^1:m'_j=m_j \;\forall\;1\leq j<i_1 \text{ and }m'_{i_1}\geq p \Big\} .
\end{equation}

\begin{lemma}\label{Zbis}
The topology of $\mathcal{G}^1$ is generated by the singletons in $\Z\times\mathcal{R}_{\mathrm{tail}}^{\mathrm{finite}}$ and by the sets \eqref{eq:thesetZk0}.
\end{lemma}

\begin{proof}
The topology of the graph groupoid is generated by the sets $Z(\alpha,\beta)$ in \eqref{eq:basicopenZ}. We now transform these sets with the composition $\Psi$ of isomorphisms $E^\infty\longrightarrow E^\infty_\ell\stackrel{\Phi}{\longrightarrow} \overline{\N}_{\nearrow}^n$ of Sect.~\ref{sec:viathemap}. Recall that $\Phi(x)_i$ counts the loops in a path $x$ before reaching the vertex $i+1$, and $\Psi$ is $\Phi$ composed with the operation that adds a loop at the head of the path. Thus, $\Psi(x)_i$ counts the number of loops in $x$ before reaching the vertex $i+1$, and adds $1$ to this number if $i\geq i_0$, where $i_0$ is the source of $x$.

Let $\alpha,\beta$ be finite paths that end at the vertex $n+1$.
The corresponding set $Z(\alpha,\beta)$ is a singleton, since $\alpha$ and $\beta$ can only be extended to right-infinite paths by adding an infinite cycle around the vertex $n+1$. The corresponding tuples
$\mv{m}':=G(\alpha\ell_{n+1}^\infty)$ and $\mv{m}:=G(\beta\ell_{n+1}^\infty)$ 
depend on the number of loops in the two paths around each of the first $n$ vertices, and can be arbitrary elements in $\N^n_{\nearrow}$.
Under the identification above,
the singleton $Z(\alpha,\beta)$ is transformed into the set
\[
\big\{ (|\beta|-|\alpha|, \mv{m}',\mv{m}) \big\} .
\]
For every fixed $\mv{m}'$ and $\mv{m}$, we can choose $\alpha$ and $\beta$ such that $|\beta|-|\alpha|=k_0$ is any integer (by looping around the vertex $n+1$ the appropriate number of times). This proves that all points in $\Z\times\mathcal{R}_{\mathrm{tail}}^{\mathrm{finite}}$ are open in $\mathcal{G}^1$.

We now study the set $Z(\alpha,\beta)$ when $\alpha$ and $\beta$ end at a vertex $i_1\leq n$.
They can be extended to right-infinite paths by adding any path $x$ that starts at the vertex $i_1$.
Call
\begin{align*}
(m_1+k_1,\ldots,m_{i_1-1}+k_{i_1-1},m''_{i_1},\ldots,m''_n) &:=\Psi(\alpha x) , \\
(m_1,\ldots,m_{i_1-1},m'_{i_1},\ldots,m'_n) &:=\Psi(\beta x) .
\end{align*}
The subsequences $(m_1,\ldots,m_{i_1-1})$ and $(m_1+k_1,\ldots,m_{i_1-1}+k_{i_1-1})$ are determined by $\beta$ and $\alpha$, respectively, and are arbitrary increasing sequences of natural numbers. We can complete the first sequence to a unique $\mv{m}\in\overline{\N}_{\nearrow}^n$ with anchor $i_1$, by filling the empty slots with $+\infty$.

For $i\geq i_1$, if both $m'_i$ and $m''_i$ are finite, their difference $m'_i-m''_i$ is independent of $x$ and given by the total number of loops in $\beta$ minus the total number of loops in $\alpha$.
Call this number $k_0$ (we can get any arbitrary integer number by letting $\alpha$ and $\beta$ loop around the vertex $i_1$ the appropriate number of times). If $m'_i=m''_i=\infty$, it is still true that $m'_i=m''_i+k_0$. We complete $(k_1,\ldots,k_{i_1-1})$ to an $n$-tuple by setting $k_i:=-k_0$ for all $i\geq i_1$. In this way, $m''_i=m'_i+k_i$ for all $i\geq i_1$. With this choice, $(k_0,\mv{k},\mv{m})\in\mathcal{F}^1$.

Observe that, if $p\in\N$ is the total number of loops in $\beta$, then
\[
m_{i_1}'\geq p+1\geq m_{i_1-1} .
\]
Finally,
\[
|\beta|-|\alpha|=k_0+\delta
\]
where $\delta$ is the offset.

The set $Z(\alpha,\beta)$ becomes the set of all tuples $(k_0+\delta,\mv{k}+\mv{m}',\mv{m}')$ such that
$m'_j=m_j$ for $j<i_1$ and $m'_{i_1}\geq p+1$. After a reparametrization $p+1\to p$, we get the set \eqref{eq:thesetZk0}.
\end{proof}

\begin{prop}\label{thm:final}
The map
\begin{equation}\label{algisoSnq}
\Z\times (\Z^n\ltimes\overline{\Z}^n)|_{\overline{\N}_{\nearrow}^n}\longrightarrow \mathcal{G}^1 , \qquad
\big(k_0,\mv{k},\mv{m}\big)\longmapsto\big(k_0+\iota(\mv{m})-\iota(\mv{k}+\mv{m}),\mv{k}+\mv{m},\mv{m}\big)
\end{equation}
induces an isomorphism $\mathcal{F}\to\mathcal{G}$ of topological groupoids.
\end{prop}

\noindent
Note that without the offset term, \eqref{algisoSnq} would be a restriction and corestriction of the map \eqref{eq:Zcan}.

\begin{proof}
Call $F$ the map \eqref{algisoSnq}. This map is a composition
\begin{equation}\label{eq:decompose}
\Z\times (\Z^n\ltimes\overline{\Z}^n)|_{\overline{\N}_{\nearrow}^n}
\xrightarrow{\quad G\quad}\mathcal{G}^1
\xrightarrow{\quad H\quad}\mathcal{G}^1 ,
\end{equation}
where
\[
G(k_0,\mv{k},\mv{m}):=(k_0,\mv{k}+\mv{m},\mv{m})
\]
is a restriction and corestriction of $\id_{\Z}\times\mathrm{can}$, and
\[
H(k_0,\mv{m}',\mv{m}):=\big(k_0+\iota(\mv{m})-\iota(\mv{m}'),\mv{m}',\mv{m}\big) .
\]
The map $G$ is a surjective homomorphism of groupoids. The map $H$ is bijective, and is a homomorphism of groupoids, since it does not change source and target of a morphism, is the identity on units, and transforms compositions into compositions:
\begin{align*}
H\big((k_0',\mv{m}'',\mv{m}')\cdot (k_0,\mv{m}',\mv{m})\big)
&=H(k_0'+k_0,\mv{m}'',\mv{m}) \\
&=\big(k_0'+k_0+\iota(\mv{m})-\iota(\mv{m}''),\mv{m}'',\mv{m}\big) \\
&=\big(k_0'+k_0+\iota(\mv{m})-\iota(\mv{m}')+\iota(\mv{m}')-\iota(\mv{m}''),\mv{m}'',\mv{m}\big) \\
&=\big(k_0'+\iota(\mv{m}')-\iota(\mv{m}''),\mv{m}'',\mv{m}'\big)\cdot
\big(k_0+\iota(\mv{m})-\iota(\mv{m}'),\mv{m}',\mv{m}\big) \\
&=H(k_0',\mv{m}'',\mv{m}')\cdot H(k_0,\mv{m}',\mv{m}) .
\end{align*}
Thus, the composition \eqref{eq:decompose} is a (surjective) homomorphism of groupoids as claimed.

If $\mv{m}\in\N_{\nearrow}^n$ has anchor $n+1$, the fiber of $G$ over $(k_0,\mv{k}+\mv{m},\mv{m})$ consists of the single point $(k_0,\mv{k},\mv{m})$, which belongs to $\mathcal{F}^1$. On the other hand, if $\mv{m}$ has anchor $i_1\leq n$, then
\[
G^{-1}(k_0,\mv{k}+\mv{m},\mv{m})=\Big\{
(k_0,\mv{k}',\mv{m}):
k_1'=k_1,\ldots,k_{i_1-1}'=k_{i_1-1}
\Big\} .
\]
The condition \eqref{eq:fixesbis} fixes the value of $\mv{k}'$, so that the restriction of $G$ to $\mathcal{F}^1$ is a bijection,
and the restriction $F=H\circ G:\mathcal{F}^1\to\mathcal{G}^1$ is an isomorphism of groupoids.
We want to show that it is also a homeomorphism, and so an isomorphism of \emph{topological} groupoids.

Clearly $(k_0,\mv{k},\mv{m})\in\mathcal{F}^1$ has no infinities if and only if $F(k_0,\mv{k},\mv{m})$ has no infinities. Thus, $F$ transforms open points of $\mathcal{F}^1$ into open points of $\mathcal{G}^1$.
On the other hand, if $i_1:=\varepsilon(\mv{m})\leq n$, a neighborhood basis of $(k_0,\mv{k},\mv{m})$ is obtained by intersecting $\mathcal{F}^1$ with the set
\[
\{k_0\}\times\{k_1\}\times\ldots\times\{k_n\}\times\{m_1\}\times\ldots\times\{m_{i_1-1}\}\times\interval{p}{+\infty}^{n-i_1+1}
\]
for every $p\geq m_{i_1-1}$. The map $F$ transforms this set into the set $Z(k_0,\mv{k},\mv{m},p)$ given by \eqref{eq:thesetZk0}. Thus, it transforms a (topological) basis of $\mathcal{F}^1$ into a (topological) basis of $\mathcal{G}^1$.
\end{proof}

As a corollary, we get an independent proof of \cite[Thm.~3.1]{She97}.

\begin{cor}[\protect{\cite[Thm.~3.1]{She97}}]
$C(S^{2n+1}_q)\cong C^*(\mathfrak{F}_n)$, where $\mathfrak{F}_n$ is the groupoid in Example~\ref{ex:sheu}.
\end{cor}

\smallskip

\begin{center}
\textsc{Acknowledgements}
\end{center}

\noindent
This research is part of the EU Staff Exchange project 101086394 ``Operator Algebras That One Can See''.

\end{document}